\documentclass[a4paper,11pt]{article}
\usepackage{amsmath,amssymb,amsthm}
\title{\bf Real-Part Estimates for Solutions of the Riesz System in $\mathbb{R}^3$\footnote{The present article is a preliminary version, submitted to Complex Variables and Elliptic Equations an International Journal.}}
\author{J. Morais\thanks{Bauhaus-Universit\"at Weimar, Institut f\"ur Mathematik/Physik, Coudraystr. 13B, D-99421 Weimar, Germany. Email: jmorais@mat.ua.pt} \, and K. G\"urlebeck\thanks{Bauhaus-Universit\"at Weimar, Institut f\"ur Mathematik/Physik, Coudraystr. 13B, D-99421 Weimar, Germany. Email: klaus.guerlebeck@uni-weimar.de}}
\date{}
\newtheorem{Theorem}{Theorem}[section]
\newtheorem{Lemma}{Lemma}[section]
\newtheorem{Definition}{Definition}[section]
\newtheorem{Remark}{Remark}[section]
\newtheorem{Proposition}{Proposition}[section]

\newtheorem{Corollary}{Corollary}[section]
\usepackage{amsfonts}
\begin{document}
\maketitle
\begin{abstract}
Main goal of this paper is to generalize Hadamard's real part theorem and invariant forms of Borel-Carath\'eodory's theorem from complex analysis to solutions of the Riesz system in the three-dimensional Euclidean space in the framework of quaternionic analysis.
\end{abstract}

\noindent {\bf Keywords}: {\small Quaternionic analysis, Riesz System, Legendre polynomials, Chebyshev polynomials, Hadamard's real part theorem, Borel-Carath\'eodory's theorem.}

\noindent {\bf MSC Subject-Classification}: 30G35, 30A10

\section{Introduction}

The study of estimates for complex analytic functions and their derivatives is rich from a purely mathematical point of view and provides an indispensable and powerful tool for solving a wide range of problems arising in applications. The estimates in matter, known as "real part theorems", attracted the attention of many complex analysts since the classical Hadamard's real part theorem (1892). Its great importance in complex analysis is our main motivation to look for a plausible higher dimensional analogue in the context of quaternionic analysis.

The analysis of the growth and local behavior of holomorphic functions in one and several complex variables occupies a special place in complex function theory and has a wide range of applications. Concrete applications arise, for instance, in complex dynamics, boundary value problems and partial differential equations theory or other fields of physics and engineering. Excellent contributions to this subject have been made, in particular, by Hadamard \cite{Hadamard1892}, Landau \cite{Landau1906}, Wiman \cite{Wiman1914}, Jensen \cite{Jensen1919}, Koebe \cite{Koebe1920}, Borel \cite{Borel1922}, M. Riesz \cite{Riesz1927}, Littlewood \cite{Littlewood1947}, Titchmarsh \cite{Titchmarsh1949}, Rajagopal \cite{Rajagopal1953}, Elkins \cite{Elkins1971}, Holland \cite{Holland1973}, Hayman \cite{Hayman1974}, Levin \cite{Levin1996} and Kresin and Maz'ya \cite{KresinVladimir2007}. Being such a classical object of analysis it became recently more and more important to perform an analogous study in higher dimensions and/or for other partial differential equations. One natural higher dimensional generalization of complex analysis can be realized by the theory of (monogenic) functions with values in a Clifford algebra that satisfy generalized Cauchy-Riemann or Dirac systems. This approach is nowadays called Clifford analysis. To aid the reader, see \cite{BDS1982,Del1970,DSS1992,DSS21992,Del2001} for more complete accounts of this subject and related topics. We restrict ourselves in 
the present paper to the case of monogenic functions defined in a ball of $\mathbb{R}^3$ with values in the reduced quaternions (identified with $\mathbb{R}^3$). This class of functions coincides with the solutions of the well known Riesz system and shows more analogies to complex holomorphic functions than the more general class of quaternion-valued functions.

According to our current knowledge many central questions concerning the growth and local behavior within the framework of Clifford analysis remain untouched so far. In a series of papers \cite{GueJoaoHadamard2008,GueJoaoPaula2009,GueJoaoMappings2009} the authors obtained various pointwise inequalities for a class of monogenic functions in a ball with the scalar part of the function on the right-hand side of the inequality. We refer to these inequalities as "real part theorems" in honor to the first assertion of such a kind, the classical (improved) Hadamard's real part theorem. A lot of results and extended list of references concerning these and other types of inequalities for monogenic functions, as well as their different modifications, can be found in \cite{JoaoThesis2009} Ch.3. The main disadvantage of the mentioned papers was that for a monogenic function defined in a ball of radius $r$ the real part estimates could be only proved in a ball of radius $r/2$. Here, we get these estimates with improved constants in the whole domain of the monogenic function. 

The paper is organized as follows. After some preliminaries, Section 3 studies a convenient orthogonal polynomial system of homogeneous monogenic polynomials defined in $\mathbb{R}^3$ with values in the reduced quaternions, deeply exploited in \cite{JoaoThesis2009}. It will be shown that this system can be viewed as analogue to the complex case of the Fourier exponential functions $\{e^{i n \theta}\}_{n \geq 0}$ on the unit circle, related to the well known Legendre and Chebyshev polynomials. Based on this system we proceed with the construction of higher dimensional counterparts for Hadamard type inequalities and of Borel-Cara\-th\'eodory inequalities for monogenic functions with values in the reduced quaternions. As a consequence, we obtain asymptotic relations between the Euclidean norms of a reduced quaternion-valued monogenic function $f$ and its scalar part, as well as between the Euclidean norms of the (hypercomplex) derivative of $f$ and the scalar part of $f$. In addition to that, estimates for a reduced quaternion-valued monogenic function in terms of various characteristics of its scalar part are also established.

\section{Basic notions}

As it is well known, a holomorphic function $f(z) = u(x,y) + iv(x,y)$ defined in an open domain of the complex plane, satisfies the Cauchy-Riemann system
\begin{eqnarray*}
\left\{ \begin{array} {ccc}
\partial_x u &=& \partial_y v \\
\partial_y u &=& - \partial_x v
\end{array}\right..
\end{eqnarray*}

As in the case of two variables, we may characterize the analogue of the Cauchy-Riemann system in an open domain of the Euclidean space $\mathbb{R}^3$. More precisely, consider a vector-valued function $f^*=(u_0,u_1,u_2)$ whose components $u_i = u_i(x_0,x_1,x_2)$ are real functions of real variables $x_0, x_1, x_2$ for which
\begin{eqnarray*}
\left\{ \begin{array} {ccc}
\hspace{0.95cm} \displaystyle \sum_{i=0}^2 \partial_{x_i} u_i = 0 && \\
\partial_{x_j} u_i - \partial_{x_i} u_j = 0 \hspace{-0.05cm} && \hspace{-0.05cm} (i \neq j, \,\, 0 \leq i, j \leq 2)
\end{array}\right.
\end{eqnarray*}
or, equivalently, in a more compact form:
\begin{eqnarray} \label{RieszSystem}
\left\{ \begin{array} {ccc}
{\rm div} \hspace{0.1cm} f^* &=& 0 \\
{\rm rot} \hspace{0.1cm} f^* &=& 0
\end{array}\right..
\end{eqnarray}

This $3$-tuple $f^*$ is said to be a system of conjugate harmonic functions in the sense of Stein-Weiss \cite{SteinWeiss1968}, and system $(\ref{RieszSystem})$ is called the Riesz system, which physically describes the velocity field of a stationary flow of a non-compressible fluid without sources and sinks.

\bigskip

The system (\ref{RieszSystem}) can be obtained naturally by working with the quaternionic algebra. Let $\mathbb{H}:=\left\{\textbf{a}=a_0 + a_1 \textbf{e}_1+a_2 \textbf{e}_2+a_3 \textbf{e}_3, \, a_i \in \mathbb{R}, \, i=0,1,2,3\right\}$ be the algebra of the real quaternions, where the imaginary units $\textbf{e}_i$ ($i=1,2,3$) are subject to the
multiplication rules
\begin{eqnarray*} 
&& \textbf{e}_1^2 = \textbf{e}_2^2 = \textbf{e}_3^2 = -1, \\
&& \textbf{e}_1 \textbf{e}_2 = \textbf{e}_3 = - \textbf{e}_2 \textbf{e}_1, \,\,\,\,\,
\textbf{e}_2 \textbf{e}_3 = \textbf{e}_1 = - \textbf{e}_3 \textbf{e}_2, \,\,\,\,\,
\textbf{e}_3 \textbf{e}_1 = \textbf{e}_2 = - \textbf{e}_1 \textbf{e}_3 .
\end{eqnarray*}

The real number $\mathbf{Sc}(\mathbf{a}) := a_0$ and the vector $\textbf{Vec}(\mathbf{a}) := a_1 \mathbf{e}_1 + a_2 \mathbf{e}_2 + a_3 \mathbf{e}_3$ are the scalar and vector part of $\mathbf{a}$, respectively. Analogously to the complex case, the conjugate of $\mathbf{a}$ is the quaternion $\overline{\mathbf{a}} = a_0 - a_1 \textbf{e}_1 - a_2 \textbf{e}_2 - a_3 \textbf{e}_3$. The norm of $\mathbf{a}$ is given by $|\mathbf{a}| = \sqrt{\mathbf{a} \overline{\mathbf{a}}}$, and coincides with its corresponding Euclidean norm, as a vector in $\mathbb{R}^4$.

\bigskip

In what follows we consider the subset $\mathcal{A} := {\rm span}_{\mathbb{R}}\{1,\mathbf{e}_1,\mathbf{e}_2\}$ of $\mathbb{H}$. Its elements are called reduced quaternions. The real vector space $\mathbb{R}^3$ is embedded in $\mathcal{A}$ via the identification
\begin{eqnarray*}
x := (x_0,x_1,x_2) \in \mathbb{R}^3  \,\,\,  \leftrightarrow \,\,\, \textbf{x} := x_0 + x_1 \textbf{e}_1 + x_2 \textbf{e}_2 \in \mathcal{A}.
\end{eqnarray*}

We denote by $\underline{x}$ the vectorial part of the reduced quaternion $\textbf{x}$, that 
is, $\underline{x} := x_1 \textbf{e}_1 + x_2 \textbf{e}_2$. Also, we emphasize that $\mathcal{A}$ is a real vectorial subspace, but not a sub-algebra, of $\mathbb{H}$.

Let $\Omega$ be an open subset of $\mathbb{R}^3$ with a piecewise smooth boundary. A reduced quaternion-valued function or, briefly, an $\mathcal{A}$-valued function is a mapping $\mathbf{f} : \Omega \longrightarrow \mathcal{A}$ such that
\begin{eqnarray*}
\mathbf{f}(x) = [\mathbf{f}(x)]_0 + \sum_{i=1}^{2} [\mathbf{f}(x)]_i \textbf{e}_i, \,\,\, x \in \Omega,
\end{eqnarray*}
where the coordinate-functions $[\mathbf{f}]_i$ $(i=0,1,2)$ are real-valued functions defined in $\Omega$. Properties such as continuity, differentiability or integrability are ascribed coordinate-wise. We will work with the real-linear Hilbert space of square integrable $\mathcal{A}$-valued functions defined in $\Omega$, that we denote by $L_2(\Omega;\mathcal{A};\mathbb{R})$. The real-valued inner product is defined by
\begin{eqnarray} \label{InnerProduct}
<\mathbf{f},\mathbf{g}>_{L_2(\Omega;\mathcal{A};\mathbb{R})} \, = \int_{\Omega}\;{\textbf{Sc}}({\overline{\mathbf{f}}\,\mathbf{g}) \, dV} \,,
\end{eqnarray}
where $dV$ denotes the Lebesgue measure in $\mathbb{R}^3$. We denote further by $B_R := B_R(0)$ the ball of radius $R$ in $ \mathbb{R}^3$ centered at the origin and $S_R = \partial B_R$ its boundary. We remark that in the case $R=1$ we omit $R$ in the notations.

For continuously real-differentiable functions $\mathbf{f}:\Omega \longrightarrow \mathcal{A}$, we consider the (reduced) quaternionic operator
\begin{eqnarray} \label{CauchyRiemannoperator}
D = \partial_{x_0} + \textbf{e}_1\,\partial_{x_1} + \textbf{e}_2\,\partial_{x_2}
\end{eqnarray}
which is called generalized Cauchy-Riemann operator on $\mathbb{R}^3$. In the same way, we define the conjugate quaternionic Cauchy-Riemann operator by
\begin{eqnarray} \label{conjugateCauchyRiemannoperator}
\overline{D} = \partial_{x_0} - \textbf{e}_1\,\partial_{x_1} - \textbf{e}_2\,\partial_{x_2}.
\end{eqnarray}

The main objects of study in quaternionic analysis are the so-called monogenic functions which may be broadly described as null-solutions of the generalized Cauchy-Riemann operator $D$. More precisely, a continuously real-differentiable $\mathcal{A}$-valued function $\mathbf{f}$ in $\Omega$ is called monogenic in $\Omega$ if it fulfills the equation $D\mathbf{f}=0$ in $\Omega$.

\begin{Remark}
For a continuously real-differentiable scalar-valued function the application of the operator $D$ coincides with the usual gradient, 
$\rm{grad}$.
\end{Remark}

\begin{Definition}
A continuously real-differentiable $\mathcal{A}$-valued function $\mathbf{f}$ is called anti-monogenic in $\Omega$ if 
$\overline{D} \mathbf{f}=0$ in $\Omega$.
\end{Definition}

\begin{Remark}
It is not necessary to distinguish between left and right monogenic (resp., anti-monogenic) in the case of $\mathcal{A}$-valued functions because $D\mathbf{f} = 0$ (resp., $\overline{D} \mathbf{f}=0$) implies $\mathbf{f} D = 0$ (resp., 
$\mathbf{f} \overline{D}=0$) and vice versa.
\end{Remark}

\begin{Remark}
For $\mathbf{f}\in C^1(\Omega;{\mathcal A})$ it holds $\overline{D} \mathbf{f}=0 \, \Leftrightarrow \, D \overline{\mathbf{f}} = 0$. This property is analogous to the complex case and is very different from the general situation of $\mathbb{H}$-valued monogenic functions. In our point of view this is one of the arguments to accept the class of $\mathcal{A}$-valued monogenic functions as a good generalization of the class of holomorphic functions in the plane.
\end{Remark}

\begin{Remark}
The generalized Cauchy-Riemann operator $(\ref{CauchyRiemannoperator})$ and its conjugate $(\ref{conjugateCauchyRiemannoperator})$ factorize the $3$-dimen\-sional Laplace operator, in the sense that $\Delta_3 \mathbf{f} = D \overline{D} \mathbf{f} = \overline{D} D \mathbf{f}$, whenever $\mathbf{f} \in C^2$, which implies that any monogenic function is also a harmonic function. This factorization of the Laplace operator establishes a special relationship between quaternionic analysis and harmonic analysis wherein monogenic functions refine the properties of harmonic functions.
\end{Remark}

\begin{Definition} {\rm (Hypercomplex derivative, see \cite{GueMal1999,MitelmanShapiro1995,Sud1979})}
Let $\mathbf{f}$ be a continuously real-differentiable $\mathcal{A}$-valued monogenic function in $\Omega$. The expression $(\frac{1}{2} \overline{D}) \mathbf{f}$ is called hypercomplex derivative of $\mathbf{f}$ in $\Omega$.
\end{Definition}

\begin{Definition} {\rm (Hyperholomorphic constant)}
A continuously real-differen\-tiable $\mathcal{A}$-valued monogenic function $\mathbf{f}$ with an identically vanishing hypercomplex derivative is called hyperholomorphic constant.
\end{Definition}

Consider now the vector field $\mathbf{f}^*=([\mathbf{f}]_0,-[\mathbf{f}]_1,-[\mathbf{f}]_2)$ associated to an $\mathcal{A}$-valued function $\mathbf{f}$. The Riesz system (\ref{RieszSystem}) can be written as
\begin{eqnarray*}
{\rm (R)} \,\,\, \left\{ \begin{array} {ccc}
\displaystyle \partial_{x_0} [\mathbf{f}]_0 - \sum_{i=1}^2 \partial_{x_i} [\mathbf{f}]_i = 0 && \\
\hspace{0.63cm} \partial_{x_j} [\mathbf{f}]_i + \partial_{x_i} [\mathbf{f}]_j = 0 && (i \neq j, \,\, 0 \leq i, j \leq 2)
\end{array}\right.
\end{eqnarray*}
or, in a more compact form as $D\mathbf{f} = \mathbf{f}D = 0$.

Following \cite{Leutwiler2001}, the solutions of the system (R) are called (R)-solutions. The subspace $L_2(\Omega; \mathcal{A}; \mathbb{R})\cap \ker D$ of polynomial (R)-solutions of degree $n$ is denoted by $\mathcal{M}^+(\Omega;\mathcal{A};n)$. In \cite{Leutwiler2001}, it is shown that the space $\mathcal{M}^+(\Omega;\mathcal{A};n)$ has dimension $2n+3$. We denote further by $\mathcal{M}^+(\Omega;\mathcal{A})=L_2(\Omega;\mathcal{A};\mathbb{R})$ the space of square integrable $\mathcal{A}$-valued monogenic functions defined in $\Omega$.

\section{A complete orthonormal set of homogeneous \\ polynomial solutions of the Riesz system in $\mathbb{R}^3$}

As is well-known, the Taylor expansion of a complex-valued holomorphic function in the unit disk $\mathbb{D} \subset \mathbb{C}$ can be seen also as its Fourier expansion with respect to the orthonormal system $\left\{\sqrt{\frac{n+1}{\pi}} z^n \right\}_{n \in \mathbb{N}}$, where the complex powers $z^n$ are functions of the real variables $x$ and $y$ satisfying the properties:
\begin{enumerate}
\item The functions $(x + iy)^n$ are homogeneous holomorphic polynomials;
\item Homogeneous holomorphic polynomials of different order are orthogonal in $L_2(\mathbb{D})$;
\item The real parts of the basis functions are orthogonal in $L_2(\mathbb{D})$;
\item The real and imaginary parts of the basis functions are mutually orthogonal in $L_2(\mathbb{D})$;
\item It holds $\partial_z z^n = n z^{n-1}$, i.e., the differentiation of the basis functions delivers again basis functions.
\end{enumerate}

While in the complex case the characterization of a holomorphic function by its Taylor or Fourier series expansions seems to be the same in principle, these two series expansions are essentially different in the quaternionic case. The reason is that the Taylor expansion with respect to the Fueter variables ${\bf z}_i=x_i-x_0{\bf e}_i, i=1,2$ does in general not give orthogonal summands (see \cite{BDS1982} and  \cite{JoaoThesis2009} Ch.2 for a special approach). If the aforementioned ideas are to be adapted to the quaternionic setting then the natural choice is a convenient system of homogeneous monogenic polynomials that will replace the previous system in the case of a Fourier series expansion and such that the previous properties are also regarded. We realize that the discussion about such an extension has originated many questions. A broad description of the most of the different approaches taken in this work together with other related results can be found in the books \cite{BDS1982,GHS2008} and the  references therein.

In (\cite{DissCacao2004}, Ch.3) and \cite{IGS2006}, special $\mathbb{R}$-linear and $\mathbb{H}$-linear complete orthonormal systems of $\mathbb{H}$-valued homogeneous monogenic polynomials defined in the unit ball of $\mathbb{R}^3$ are explicitly constructed ($\mathbb{R}^3 \rightarrow \mathbb{R}^4$ case). The main idea of such constructions is the already referred factorization of the Laplace operator by $D \overline{D}$. Partially motivated by these results we deal  in the following with a convenient system of polynomial solutions of the Riesz system ($\mathbb{R}^3 \rightarrow \mathbb{R}^3$ case), exploited in \cite{JoaoThesis2009}. It will be shown that this system is analogous to the complex case of the Fourier exponential functions $\{e^{i n \theta}\}_{n \geq 0}$ on the unit circle and constitutes an extension of the role of the well known Legendre and Chebyshev polynomials.

For a detailed description we  introduce spherical coordinates,
\begin{equation*}
x_0 = r \cos \theta_1, \,\, x_1 = r \sin \theta_1 \cos \theta_2, \,\, x_2 = r \sin \theta_1 \sin \theta_2,
\end{equation*} where
$0 < r < \infty$, $0 < \theta_1 \leq \pi$, $0 < \theta_2 \leq 2\pi$. Each point $x=(x_0,x_1,x_2) \in \mathbb{R}^3 \backslash \{0\}$ admits a unique representation $x=r \omega$, where for each $i=0,1,2$ $\omega_i = \frac{x_i}{r}$ and $|\omega|=1$. The strategy adopted is the following: we start by considering the set of homogeneous harmonic polynomials
\begin{equation} \label{HHP}
\{ r^{n+1} U^l_{n+1}, r^{n+1} V^m_{n+1}, \, l=0,\ldots,n+1, \, m=1,\ldots,n+1  \}_{n \in \mathbb{N}_0}
\end{equation}
formed by the extensions in the ball of the 2n+3 linearly independent spherical harmonics
\begin{eqnarray} \label{SphericalHarmonics}
U^l_{n+1}(\theta_1,\theta_2) &=& P^l_{n+1}(\cos \theta_1) \, T_l(\cos \theta_2), \hspace{1.12cm} \qquad l=0,\ldots,n+1 \\
V^m_{n+1}(\theta_1,\theta_2) &=& P^m_{n+1}(\cos \theta_1) \sin \theta_2 \, U_{m-1}(\cos \theta_2), \,\, m=1,\ldots,n+1 \nonumber.
\end{eqnarray}

Here, $P^l_{n+1}$ stands for the associated Legendre functions of degree $n+1$, $T_l$ and $U_{m-1}$ are the Chebyshev polynomials of the first and second kinds, respectively. We remark that whenever $l=0$, the corresponding associated Legendre function $P^0_{n+1}(t)$ coincides with the Legendre polynomial $P_{n+1}(t)$.

\begin{Remark}
Throughout the paper we follow the notations of \cite{San1959} where $\{ U^l_n, V^m_n : l=0,\ldots,n, \, m=1,\ldots,n \}$ is 
used for the orthogonal basis $(\ref{SphericalHarmonics})$ of spherical harmonics in $\mathbb{R}^3$. Even though there are only a few differences between the notations of $U^l_n$ and the Chebyshev polynomials of the second kind, $U_{m-1}$, we will keep them in the following.
\end{Remark}

For the associated Legendre functions several estimates can be found in the literature. From (e.g., \cite{Lohofer1998} p.179) we get the upper bound
\begin{eqnarray} \label{estimatesLegendreFunctions}
\max_{t \in [-1,1]} |P^m_n(t)| \,\leq\, \frac{1}{\sqrt{2}} \sqrt{\frac{(n+m)!}{(n-m)!}},
\end{eqnarray}
which is valid for $m=1,\ldots,n$. 

A more detailed study of the Legendre polynomials and their associated Legendre functions can be found, for example, in \cite{Andrews1998} and \cite{San1959}.

\medskip

We now come to the aim of this section. As described, we apply for each $n \in \mathbb{N}_0$, the operator
$\frac{1}{2}\overline{D}$ to the homogeneous harmonic polynomials in $(\ref{HHP})$. We obtain then the following set of homogeneous monogenic polynomials
\begin{equation} \label{HMP}
\{ \mathbf{X}^{l,\dagger}_n, ~ \mathbf{Y}^{m,\dagger}_n : \, l=0,\ldots,n+1, \, m=1,\ldots,n+1 \},
\end{equation}
with the notations 
\begin{eqnarray*}
\mathbf{X}^{l,\dagger}_n := r^n \mathbf{X}^l_n, \,\,\,\,\,\,\,\,\,\,
\mathbf{Y}^{m,\dagger}_n := r^n \mathbf{Y}^m_n.
\end{eqnarray*}

Although the polynomials $\mathbf{X}^{0,\dagger}_n$ are built in terms of the Legendre polynomials while $\mathbf{X}^{m,\dagger}_n$ are built in terms of the associated Legendre functions, we will still include the treatment of the first into the general case whenever this treatment remains the same. 

The explicit expressions of the mentioned polynomials are given shortly by
\begin{Lemma} \label{FormulaeHMP}
The homogenous monogenic polynomials $(\ref{HMP})$ can be represented in the following way
\begin{eqnarray*}
&& \mathbf{X}^{l,\dagger}_n \, := \, r^n \left[\frac{(n+l+1)}{2}  P^l_n(\cos\theta_1) \, T_l(\cos \theta_2) \right. \\
&+& \frac{1}{4} P^{l+1}_n(\cos\theta_1) \left[ T_{l+1}(\cos \theta_2) \mathbf{e}_1 + \sin \theta_2 \, U_l(\cos \theta_2) \mathbf{e}_2 \right] \\
&+& \left. \frac{1}{4} (n+l+1)(n+l) P^{l-1}_n(\cos\theta_1) \left[ - T_{l-1}(\cos \theta_2) \mathbf{e}_1 
+ \sin \theta_2 \, U_{l-2}(\cos \theta_2) \mathbf{e}_2 \right] \right]
\end{eqnarray*}
and
\begin{eqnarray*}
&& \mathbf{Y}^{m,\dagger}_n \, := \, r^n \left[\frac{(n+m+1)}{2}  P^m_n(\cos\theta_1) \sin \theta_2 \, U_{m-1}(\cos \theta_2) \right. \\
&+& \frac{1}{4} P^{m+1}_n(\cos\theta_1) \left[ \sin \theta_2 \, U_m(\cos \theta_2) \mathbf{e}_1 
- T_{m+1}(\cos \theta_2) \mathbf{e}_2 \right] \\
&-& \left. \frac{1}{4} (n+m+1)(n+m) P^{m-1}_n(\cos\theta_1) \left[ \sin \theta_2 \, U_{m-2}(\cos \theta_2) \mathbf{e}_1 
+ T_{m-1}(\cos \theta_2) \mathbf{e}_2 \right] \right]
\end{eqnarray*}
where $l=0,\ldots,n+1$ and $m=1,\ldots,n+1$. For a more unified formulation we remind the reader that for the case $l=0$ the associated Legendre polynomial $P^{-1}_n$ is defined by $P^{-1}_n = -\frac{1}{n(n+1)} P^1_n$ and the associated Legendre functions $P^i_n$ are the zero function for 
$i \geq n+1$.
\end{Lemma}
\begin{proof}
The proof consists of simple, but very lengthy calculations combining the results from \cite{JoaoThesis2009} p.24, several recurrence properties of the Legendre polynomials and their associated Legendre functions and some equations between the Chebyshev polynomials of the first and second kinds.
\end{proof}

Taking a closer look at the structure of the corresponding formulae obtained beforehand we observe that the coordinates of the constructed basis polynomials are products of the Legendre functions (resp. Legendre polynomials) and Chebyshev polynomials. If so, correspondingly one can also show that these coordinates must have simple expressions in spherical harmonics. The next proposition, a direct consequence of Lemma \ref{FormulaeHMP}, shows these relations. Besides their obvious importance these formulae also permit to determine orthogonal relations between the coordinates of the corresponding polynomials in a direct way but they are not needed in the present article and will not be further discussed here.

\begin{Proposition} \label{HMPSphericalHarmonics}
The homogenous monogenic polynomials $(\ref{HMP})$ can be represented in the following way
\begin{eqnarray*}
\mathbf{X}^{0,\dagger}_n &:=& r^n \left[\frac{(n+1)}{2} U^0_n + \frac{1}{2} U^1_n \mathbf{e}_1 
+ \frac{1}{2} V^1_n \mathbf{e}_2 \right] \\
\mathbf{X}^{m,\dagger}_n &:=& r^n \left[ \frac{(n+m+1)}{2} U^m_n 
+ \frac{1}{4} R^{m,-}_n \mathbf{e}_1
+ \frac{1}{4} S^{m,+}_n \mathbf{e}_2 \right]
\end{eqnarray*}
\begin{eqnarray*}
\mathbf{Y}^{m,\dagger}_n &:=& r^n \left[ \frac{(n+m+1)}{2} V^m_n 
+ \frac{1}{4} S^{m,-}_n \mathbf{e}_1
- \frac{1}{4} R^{m,+}_n \mathbf{e}_2 \right]
\end{eqnarray*}
with the notations
\begin{eqnarray*}
R^{m,\pm}_n (\theta_1,\theta_2) &:=& U^{m+1}_n \pm (n+m+1)(n+m) U^{m-1}_n \\
S^{m,\pm}_n (\theta_1,\theta_2) &:=& V^{m+1}_n \pm (n+m+1)(n+m) V^{m-1}_n,
\end{eqnarray*}
where $m=1,\ldots,n+1$. For a more unified formulation we remind the reader that the spherical harmonics $U^m_n$ and $V^m_n$ are the zero function for $m \geq n+1$.
\end{Proposition}

In view of the practical applicability of the polynomial system $(\ref{HMP})$, it is interesting to know whether one can efficiently obtain explicit recurrence rules which are plain to be integrated in a computational framework. Analogously to the results obtained in \cite{BockGue2009} by using a different approach, in the following we introduce the conversion formulae between the polynomials in use as a consequence of the previous proposition.

\begin{Proposition}
For each $n \in \mathbb{N}_0$, the homogenous monogenic polynomials $(\ref{HMP})$ satisfy the following recurrence formulae:
{\small{\begin{eqnarray*}
\mathbf{X}^{0,\dagger}_n &=& \frac{1}{n+2} \left( \mathbf{X}^{1,\dagger}_n \mathbf{e}_1 + \mathbf{Y}^{1,\dagger}_n \mathbf{e}_2 \right) \\
\mathbf{X}^{m,\dagger}_n &=& -\frac{n+m+1}{2} \left( \mathbf{X}^{m-1,\dagger}_n \mathbf{e}_1 - \mathbf{Y}^{m-1,\dagger}_n \mathbf{e}_2 \right) \\
&& \hspace{1.5cm} + \frac{1}{2(n+m+2)} \left( \mathbf{X}^{m+1,\dagger}_n \mathbf{e}_1+ \mathbf{Y}^{m+1,\dagger}_n \mathbf{e}_2 \right) \\
\mathbf{Y}^{m,\dagger}_n &=& -\frac{n+m+1}{2} \left( \mathbf{Y}^{m-1,\dagger}_n \mathbf{e}_1 + \mathbf{X}^{m-1,\dagger}_n \mathbf{e}_2 \right) \\
&& \hspace{1.5cm} + \frac{1}{2(n+m+2)} \left( \mathbf{Y}^{m+1,\dagger}_n \mathbf{e}_1 - \mathbf{X}^{m+1,\dagger}_n \mathbf{e}_2 \right)
\end{eqnarray*}}}
for $m=1,\ldots,n+1$, where $\mathbf{X}^{0,\dagger}_0 = \frac{1}{2} = \mathbf{X}^{1,\dagger}_0 \mathbf{e}_1 = \mathbf{Y}^{1,\dagger}_0 \mathbf{e}_2$. For a more unified formulation we remind the reader that $\mathbf{X}^{m,\dagger}_n$ and $\mathbf{Y}^{m,\dagger}_n$ are the zero function for $m \geq n+2$.
\end{Proposition}
\begin{proof}
The proof is a consequence of the previous proposition by direct inspection of the relations between the coordinates of the basis polynomials and the used spherical harmonics.
\end{proof}

\begin{Remark}
We can see in the proposition that $\mathbf{X}^{0,\dagger}_n,   \mathbf{X}^{1,\dagger}_n, \mathbf{Y}^{1,\dagger}_n $ define for each degree $n$ the initial values for the recurrence scheme. Solving the system of the second and third equation of the proposition for  $  \mathbf{X}^{m+1,\dagger}_n, \mathbf{Y}^{m+1,\dagger}_n$ we get the desired recurrence formulae.
\end{Remark}

In what follows we state some properties of the basis polynomials. We show that these polynomials preserve all properties of the powers $z$ with positive exponent described in the beginning of this section. 

\begin{Theorem} \label{PropertiesHMP}
The homogeneous monogenic polynomials $(\ref{HMP})$ satisfy the following properties:
\begin{enumerate}
\item{Given a fixed $n \in \mathbb{N}_0$, we have the relations:
\begin{eqnarray*}
\mathbf{Sc}(\mathbf{X}_n^{l,\dagger}) &:=& \frac{(n+l+1)}{2} ~ U^{l,\dagger}_{n}, \hspace{0.79cm} l=0,\ldots,n+1 \\
\mathbf{Sc}(\mathbf{Y}_n^{m,\dagger}) &:=& \frac{(n+m+1)}{2} ~ V^{m,\dagger}_{n}, \,\,\, m=1,\ldots,n+1;
\end{eqnarray*}}
\item{For each $n \in \mathbb{N}_0$, the homogeneous monogenic polynomials of the set $(\ref{HMP})$ are orthogonal in $L_2(B;\mathcal{A};\mathbb{R})$;}
\item{For each $n \in \mathbb{N}_0$, the homogeneous harmonic polynomials in each of the sets
\begin{eqnarray*}
&& \{\mathbf{Sc}(\mathbf{X}_n^{l,\dagger}), \, [\mathbf{X}_n^{l,\dagger}]_1, \, [\mathbf{X}_n^{l,\dagger}]_2: \, l=0,\ldots,n+1\} \\
&& \{\mathbf{Sc}(\mathbf{Y}_n^{m,\dagger}), \, [\mathbf{Y}_n^{m,\dagger}]_1, \, [\mathbf{Y}_n^{m,\dagger}]_2: \, m=1,\ldots,n+1\}
\end{eqnarray*}
are orthogonal in $L_2(B)$;}
\item{For each $n \in \mathbb{N}_0$, the homogeneous anti-monogenic polynomials 
\begin{eqnarray*}
&& \{{\rm grad}(\mathbf{Sc}(\mathbf{X}_n^{l,\dagger})), \, {\rm grad}(\mathbf{Sc}(\mathbf{Y}_n^{m,\dagger})): \, l=0,\ldots,n+1, \, m=1,\ldots,n+1 \}
\end{eqnarray*}
are orthogonal in $L_2(B;\mathcal{A};\mathbb{R})$;}
\item{For each $n \in \mathbb{N}_0$, the homogeneous anti-monogenic polynomials in each of the sets
\begin{eqnarray*}
&& \{{\rm grad}(\mathbf{Sc}(\mathbf{X}_n^{l,\dagger})), \, {\rm grad}([\mathbf{X}_n^{l,\dagger}]_1), 
\, {\rm grad}([\mathbf{X}_n^{l,\dagger}]_2): \, l=0,\ldots,n+1 \} \\
&& \{{\rm grad}(\mathbf{Sc}(\mathbf{Y}_n^{m,\dagger})), \, {\rm grad}([\mathbf{Y}_n^{m,\dagger}]_1), \, {\rm grad}([\mathbf{Y}_n^{m,\dagger}]_2): \, m=1,\ldots,n+1\}
\end{eqnarray*}
are orthogonal in $L_2(B;\mathcal{A};\mathbb{R})$;}
\item{For $n \geq 1$, we have the identities:
\begin{eqnarray*}
(\frac{1}{2} \overline{D}) \mathbf{X}_n^{l,\dagger} &=& (n+l+1) \mathbf{X}_{n-1}^{l,\dagger}, \hspace{0.55cm} l=0,\ldots,n+1 \\
(\frac{1}{2} \overline{D}) \mathbf{Y}_n^{m,\dagger} &=& (n+m+1) \mathbf{Y}_{n-1}^{m,\dagger}, \,\, m=1,\ldots,n+1;
\end{eqnarray*}}
\item{The polynomials $\mathbf{X}^{n+1,\dagger}_n$ and $\mathbf{Y}^{n+1,\dagger}_n$ are hyperholomorphic constants. Moreover, they are of the form
\begin{eqnarray*}
\mathbf{X}^{n+1,\dagger}_n &=& -\frac{\mathbf{e}_1}{2}(n+1)(2n+1)!! \, 
r^n \left[(\sin \theta_1)^n (\cos \theta_2 + \mathbf{e}_3 \sin \theta_2)^n \right] \\
\mathbf{Y}^{n+1,\dagger}_n &=& -\frac{\mathbf{e}_2}{2}(n+1)(2n+1)!! \,
r^n \left[ (\sin \theta_1)^n (\cos \theta_2 + \mathbf{e}_3 \sin \theta_2)^n \right].
\end{eqnarray*}}
\end{enumerate}
\end{Theorem}
\begin{proof}
Statements $1.$ and $3.$ are direct consequences of Proposition \ref{HMPSphericalHarmonics}. Statement 2. may be found in \cite{DissCacao2004}. For the proof of statement 4. we first remind the reader that the application of the gradient on a scalar-valued function is done just by applying the operator $D$. Having also in mind the strategy adopted for the construction of the basis polynomials, we obtain the following relations
\begin{eqnarray*}
{\rm grad}(\mathbf{Sc}(\mathbf{X}_n^{l,\dagger})) &:=& \frac{(n+l+1)}{2} \, \overline{\mathbf{X}_{n-1}^{l,\dagger}} \\
{\rm grad}(\mathbf{Sc}(\mathbf{Y}_n^{m,\dagger})) &:=& \frac{(n+m+1)}{2} \, \overline{\mathbf{Y}_{n-1}^{m,\dagger}},
\end{eqnarray*}
for $l=0,\ldots,n$ and $m=1,\ldots,n$, where ${\rm grad}(\mathbf{Sc}(\mathbf{X}_n^{l,\dagger}))$ and ${\rm grad}(\mathbf{Sc}(\mathbf{Y}_n^{m,\dagger}))$ are homogeneous polynomials belonging to the kernel of $\overline{D}$. We remark that ${\rm grad}(\mathbf{Sc}(\mathbf{X}_n^{m,\dagger}))$ and ${\rm grad}(\mathbf{Sc}(\mathbf{Y}_n^{m,\dagger}))$ are the zero function for $m = n+1$. Now, by definition of the real-valued inner product $(\ref{InnerProduct})$ and using the previous expressions a direct computation shows that
\begin{eqnarray*}
&& <{\rm grad}(\mathbf{Sc}(\mathbf{X}_n^{l,\dagger}))),{\rm grad}(\mathbf{Sc}(\mathbf{Y}_n^{m,\dagger})))>_{L_2(B;\mathcal{A};\mathbb{R})} 
\, = \\ 
&=& \frac{(n+l+1)(n+m+1)}{4} <\mathbf{X}_{n-1}^{l,\dagger},\mathbf{Y}_{n-1}^{m,\dagger}>_{L_2(B;\mathcal{A};\mathbb{R})}.
\end{eqnarray*}

Taking into account the orthogonality of the set $\{\mathbf{X}_n^{l,\dagger},\mathbf{Y}_n^{m,\dagger}, \,l=0,\ldots,n+1, \, m=1,\ldots,n+1\}$ with respect to the real-inner product $(\ref{InnerProduct})$ (see statement 3.) the orthogonality of the mentioned homogenous anti-monogenic polynomials turns out. We proceed to prove statement 5. For simplicity we just present the relations between the gradients of the coordinates of $\mathbf{X}^{m,\dagger}_n$ $(m=1,\ldots,n+1)$, the proofs for $\mathbf{X}^{0,\dagger}_n$ and $\mathbf{Y}^{m,\dagger}_n$ are similar. Using the same arguments as before we obtain
\begin{eqnarray*}
{\rm grad}(\mathbf{Sc}(\mathbf{X}_n^{m,\dagger})) &:=& \frac{(n+m+1)}{2} \, \overline{\mathbf{X}_{n-1}^{m,\dagger}} \\
{\rm grad}([\mathbf{X}_n^{m,\dagger}]_1) &:=& \frac{1}{4} \overline{\mathbf{X}_{n-1}^{m+1,\dagger}} 
- \frac{(n+m+1)(n+m)}{4} \overline{\mathbf{X}_{n-1}^{m-1,\dagger}} \\
{\rm grad}([\mathbf{X}_n^{m,\dagger}]_2) &:=& \frac{1}{4} \overline{\mathbf{Y}_{n-1}^{m+1,\dagger}} 
+ \frac{(n+m+1)(n+m)}{4} \overline{\mathbf{Y}_{n-1}^{m-1,\dagger}}.
\end{eqnarray*}

Again, by definition of the real-valued inner product in $L_2(B;\mathcal{A};\mathbb{R})$ and using the previous expressions a direct computation shows that
\begin{eqnarray*}
&& <{\rm grad}(\mathbf{Sc}(\mathbf{X}_n^{m,\dagger})),{\rm grad}([\mathbf{X}_n^{m,\dagger}]_1)>_{L_2(B;\mathcal{A};\mathbb{R})} 
\, = \\ &=& \frac{(n+m+1)}{8} <\mathbf{X}_{n-1}^{m,\dagger},\mathbf{X}_{n-1}^{m+1,\dagger}>_{L_2(B;\mathcal{A};\mathbb{R})} \\
&+& \frac{(n+m+1)^2(n+m)}{8} <\mathbf{X}_{n-1}^{m,\dagger},\mathbf{X}_{n-1}^{m-1,\dagger}>_{L_2(B;\mathcal{A};\mathbb{R})}
\end{eqnarray*}
\begin{eqnarray*}
&& <{\rm grad}(\mathbf{Sc}(\mathbf{X}_n^{m,\dagger})),{\rm grad}([\mathbf{X}_n^{m,\dagger}]_2)>_{L_2(B;\mathcal{A};\mathbb{R})} 
\, = \\ &=& \frac{(n+m+1)}{8} <\mathbf{X}_{n-1}^{m,\dagger},\mathbf{Y}_{n-1}^{m+1,\dagger}>_{L_2(B;\mathcal{A};\mathbb{R})} \\
&-& \frac{(n+m+1)^2(n+m)}{8} <\mathbf{X}_{n-1}^{m,\dagger},\mathbf{Y}_{n-1}^{m-1,\dagger}>_{L_2(B;\mathcal{A};\mathbb{R})} \\\\
&& <{\rm grad}([\mathbf{X}_n^{m,\dagger}]_1),{\rm grad}([\mathbf{X}_n^{m,\dagger}]_2)>_{L_2(B;\mathcal{A};\mathbb{R})} 
\, = \\ 
&=& \frac{1}{16} <\mathbf{X}_{n-1}^{m,\dagger},\mathbf{Y}_{n-1}^{m+1,\dagger}>_{L_2(B;\mathcal{A};\mathbb{R})}  \\
&+& \frac{(n+m+1)(n+m)}{16} <\mathbf{X}_{n-1}^{m+1,\dagger},\mathbf{Y}_{n-1}^{m-1,\dagger}>_{L_2(B;\mathcal{A};\mathbb{R})} \\
&-& \frac{(n+m+1)(n+m)}{16} <\mathbf{X}_{n-1}^{m-1,\dagger},\mathbf{Y}_{n-1}^{m+1,\dagger}>_{L_2(B;\mathcal{A};\mathbb{R})} \\
&-& \frac{(n+m+1)^2(n+m)^2}{16} <\mathbf{X}_{n-1}^{m-1,\dagger},\mathbf{Y}_{n-1}^{m-1,\dagger}>_{L_2(B;\mathcal{A};\mathbb{R})}.
\end{eqnarray*}

Having in mind the orthogonality of the set $\{\mathbf{X}_n^{l,\dagger},\mathbf{Y}_n^{m,\dagger}, \,l=0,\ldots,n+1, \, m=1,\ldots,n+1\}$ with respect to the real-inner product $(\ref{InnerProduct})$ the previous inner products vanish. To end the proof we remark that statements $6.$ and $7.$ are already known from \cite{DissCacao2004} because our constructed system is a subsystem of the polynomials which were constructed in \cite{DissCacao2004}. The expressions for the hyperholomorphic constants, $\mathbf{X}_n^{n+1,\dagger}$ and $\mathbf{Y}_n^{n+1,\dagger}$, are consequence of Theorem $\ref{FormulaeHMP}$ and Moivre's formula.
\end{proof}

\begin{Remark}
The surprising observation in statement 1. is that the scalar parts of the $\mathcal{A}$-valued homogeneous monogenic polynomials which were obtained by applying the $\overline{D}$ operator to scalar-valued harmonic polynomials, are strongly related to the original polynomials. This means nothing else than that the scalar parts are orthogonal to each other. In this sense, it reflects one of the most noteworthy properties of the system and it shows an immediate relationship with the classical complex function theory in the plane, where the real parts of the complex variables $z^n$ are also orthogonal to each other. The same is valid for statements 3., 4. and 5. On the other hand, statement 6. shows what happens when we calculate the hypercomplex derivatives of the basis polynomials. In fact, a simple observation shows that the hypercomplex derivatives of the referred polynomials also belong to $\mathcal{A}$ as the polynomials themselves, being a multiple of the similar polynomial one degree lower, like in the complex case. Finally, statement 7. shows that among the homogeneous monogenic polynomials $(\ref{HMP})$, there are hyperholomorphic constants.
\end{Remark}

We proceed to show that one can now easily derive from Lemma \ref{FormulaeHMP} and inequality $(\ref{estimatesLegendreFunctions})$ the following pointwise estimates of our basis polynomials. These estimates refine the ones obtained by the authors in \cite{GueJoaoBohr2}.
\begin{Proposition} \label{modulusHMP}
For $n \in \mathbb{N}_0$ the homogeneous monogenic polynomials $(\ref{HMP})$ satisfy the following inequalities:
\begin{eqnarray*}
|\mathbf{X}^{l,\dagger}_n(x)| & \leq & \frac{1}{2} (n+1) \sqrt{\frac{(n+1+l)!}{(n+1-l)!}} \, r^n \\
|\mathbf{Y}^{m,\dagger}_n(x)| & \leq & \frac{1}{2} (n+1) \sqrt{\frac{(n+1+m)!}{(n+1-m)!}} \, r^n,
\end{eqnarray*}
for $l=0,\ldots,n+1$ and $m=1,\ldots,n+1.$
\end{Proposition}
\begin{proof}
For simplicity we just present the calculations for the spherical monogenics $\mathbf{X}^l_n$, the proof of 
$\mathbf{Y}^m_n$ is similar. Using Lemma \ref{FormulaeHMP} a direct computation shows that
\begin{eqnarray*}
|\mathbf{X}^l_n(\theta_1,\theta_2)|^2 &=& \frac{(n+l+1)^2}{4}(P^l_n(\cos \theta_1))^2 (T_l(\cos \theta_2))^2 
+ \frac{1}{16}(P^{l+1}_n(\cos \theta_1))^2 \\
&+& \frac{1}{16} (n+l+1)^2 (n+l)^2 (P^{l-1}_n(\cos \theta_1))^2 \\
&-& \frac{1}{8}(n+l+1)(n+l) P^{l-1}_n(\cos \theta_1) P^{l+1}_n(\cos \theta_1) T_{2l}(\cos \theta_2) \\
&\leq& \frac{1}{4} (n+1)^2 \frac{(n+1+l)!}{(n+1-l)!}.
\end{eqnarray*}

Note that the last inequality is also based on the estimate $(\ref{estimatesLegendreFunctions})$ for the associated Legendre functions.
\end{proof}

Since some of the further results are not only restricted to the unit ball, from now on, we shall denote by $\mathbf{X}_n^{0,\dagger,\ast_R}, \mathbf{X}_n^{m,\dagger,\ast_R}, \mathbf{Y}_n^{m,\dagger,\ast_R}$ $(m=1,\ldots,n+1)$ the new normalized basis functions $\mathbf{X}_n^{0,\dagger}, \mathbf{X}_n^{m,\dagger}, \mathbf{Y}_n^{m,\dagger}$ in $L_2(B_R;\mathcal{A};\mathbb{R})$. By using a similarity transformation, from \cite{JoaoThesis2009} we state the following result:

\begin{Lemma} \label{ONSArbitraryBall}
For each $n$, the set of homogeneous monogenic polynomials
\begin{eqnarray} \label{ONS}
\left\{ \mathbf{X}_n^{0,\dagger,\ast_R}, \mathbf{X}_n^{m,\dagger,\ast_R}, \mathbf{Y}_n^{m,\dagger,\ast_R}: ~ m = 1,...,n+1 \right\}
\end{eqnarray}
forms an orthonormal basis in $\mathcal{M}^+(B_R;\mathcal{A};n)$. Consequently, 
\begin{equation*}
\left\{\mathbf{X}_n^{0,\dagger,\ast_R}, \mathbf{X}_n^{m,\dagger,\ast_R}, \mathbf{Y}_n^{m,\dagger,\ast_R}, ~ m =
1,\ldots,n+1; n=0,1,\ldots \right\}
\end{equation*}
is an orthonormal basis in $\mathcal{M}^+(B_R;\mathcal{A})$.
\end{Lemma}

Considering the orthonormal complete system $(\ref{ONS})$ then the Fourier expansion of a square integrable ${\mathcal A}$-valued monogenic function in $L_2(B_R;\mathcal{A};\mathbb{R})$ can be defined. Moreover, the Fourier series of $\mathbf{f}$ with respect to the referred system in $L_2(B_R; \mathcal{A}; \mathbb{R}) \cap \ker D$ is given by
\begin{eqnarray} \label{FourierSeries}
\mathbf{f} = \sum_{n=0}^{\infty} \left[ \mathbf{X}_n^{0,\dagger,\ast_R} a_n^{0,\ast_R} + \sum_{m=1}^{n+1}
\left(\mathbf{X}_n^{m,\dagger,\ast_R} a_n^{m,\ast_R} + \mathbf{Y}_n^{m,\dagger,\ast_R} b_n^{m,\ast_R} \right) \right],
\end{eqnarray}
where for each $n \in \mathbb{N}_0$, $a_n^{0,\ast_R}, a_n^{m,\ast_R}, b_n^{m,\ast_R} \in \mathbb{R}$ $(m=1,\ldots,n+1)$ are the associated Fourier coefficients.

\begin{Remark}
Since by construction the system $(\ref{ONS})$ forms an orthonormal basis with respect to the real-inner product $(\ref{InnerProduct})$, we stress that the coefficients $a_n^{0,\ast_R}, a_n^{m,\ast_R}$ and $b_n^{m,\ast_R}$ $(m=1,...,n+1)$ are real constants.
\end{Remark}

Based on representation $(\ref{FourierSeries})$ and according to the fact that the polynomials $\mathbf{X}_n^{n+1,\dagger,\ast}$ and $\mathbf{Y}_n^{n+1,\dagger,\ast}$ are hyperholomorphic constants (see statement 7. Theorem \ref{PropertiesHMP}), in \cite{GueJoaoBohr2} we have proved that each $\mathcal{A}$-valued monogenic function can be decomposed in an orthogonal sum of a monogenic "main part" of the function $(\mathbf{g})$ and a hyperholomorphic constant $(\mathbf{h})$. Next we formulate the result.

\begin{Lemma} {\rm (Decomposition Theorem)} \label{Tdecomposition}
A function $\mathbf{f} \in \mathcal{M}^+(B_R;\mathcal{A})$ can be decomposed into
\begin{eqnarray} \label{decomposition}
\mathbf{f} := \mathbf{f}(0) + \mathbf{g} + \mathbf{h}, 
\end{eqnarray}
where the functions $\mathbf{g}$ and $\mathbf{h}$ have the Fourier series
\begin{eqnarray*}
\mathbf{g}(x) &=& \sum_{n=1}^{\infty} \left(\mathbf{X}_n^{0,\dagger,\ast_R}(x) a_n^{0,\ast_R} + \sum_{m=1}^{n} \left[
\mathbf{X}_n^{m,\dagger,\ast_R} (x) a_n^{m,\ast_R} + \mathbf{Y}_n^{m,\dagger,\ast_R}(x) b_n^{m,\ast_R} \right] \right) \\
\mathbf{h}(\underline{x}) &=& \sum_{n=1}^{\infty} \left[ \mathbf{X}_n^{n+1,\dagger,\ast_R}(\underline{x}) a_n^{n+1,\ast_R} + \mathbf{Y}_n^{n+1,\dagger,\ast_R}(\underline{x}) b_n^{n+1,\ast_R} \right] .
\end{eqnarray*}
\end{Lemma}

\section{Estimates for $\mathcal{A}$-valued monogenic functions \\ bounded with respect to their scalar part}

The present section is connected with two classical assertions of the analytic function's theory, namely with Hadamard's real part theorem and invariant forms of Borel-Carath\'eodory's theorem on the majorant of a Taylor's series.

In the one-dimensional complex analysis much effort has been done, during the last century, regarding the study of the growth of holomorphic functions in dependence of its real part. The existent estimates provide powerful tools in the theory of entire functions. In particular, these inequalities and their variants are applied in factorization theory (see Hadamard \cite{Hadamard1892}) and in approximation of entire functions (see Elkins \cite{Elkins1971}), the study of analytic number theory (see Ingham \cite{Ingham1932}, Ch.$3$), in mathematical physics (see Maharana \cite{Maharana1978} and also \cite{CarlesonGamelin1993}) or in applications to several partial differential equations (see e.g. \cite{Hayman1974,Wiman1914}). The fundamentals in the study of different formulations of inequalities containing the real part as a majorant have been established by Jensen \cite{Jensen1919}, Koebe \cite{Koebe1920}, M. Riesz \cite{Riesz1927}, Rajagopal \cite{Rajagopal1941,Rajagopal1947,Rajagopal1953}, Littlewood \cite{Littlewood1947}, Titchmarsh \cite{Titchmarsh1949}, Holland \cite{Holland1973}, Levin \cite{Levin1996} and recently by Kresin and Maz'ya \cite{KresinVladimir2002,KresinVladimir2007}. The reference list does not claim to be complete. Further references can be found in the books \cite{Burckel1979} and \cite{KresinVladimir2007}.

Among these inequalities is the Hadamard-Borel-Carath\'eodory inequality for analytic functions in the disk $D_R=\{z \in \mathbb{C}: |z| = r < R\}$ with $f$ bounded from above 
\begin{eqnarray*}
|f(z)-f(0)| \, \leq \, \frac{2 r}{R-r}  \sup_{|\xi| < R} |\mathbf{Re} \{f(\xi)-f(0)\}|.
\end{eqnarray*}

The same class includes inequalities for derivatives at the center of $D_R$ with the real part of the function on the right-hand side of the inequality and many other related estimates (see Jensen \cite{Jensen1919}, Koebe \cite{Koebe1920}, Rajagopal \cite{Rajagopal1941,Rajagopal1947,Rajagopal1953}, Kresin and Maz'ya \cite{KresinVladimir2007}).

During the last years, generalizations of Hadamard-Borel-Carath\'eodory's inequality have been considered for holomorphic functions in domains on a complex manifold (see Aizenberg, Aytuna and Djakov \cite{AAD2000}) and its extension for analytic multifunctions (see Chen \cite{Chen2004}).  We note that recently first results on higher dimensional counterparts of Hadamard's real part theorem and its variants were obtained in the context of quaternionic analysis in \cite{GueJoaoHadamard2008,GueJoaoPaula2009,GueJoaoMappings2009} and \cite{JoaoThesis2009}. These estimates showed the possibility to generalize the real-part theorems but the quality of the results was not satisfying. The constants were very weak and the estimates could be only proved in balls of radius $r/2$.

\bigskip

In the sequel, we shall introduce the notation
\begin{eqnarray*}
\mathcal{M}(\mathbf{f},r) = \max_{|x| = r} |\mathbf{f}(x)|, \,\,\,\,\,\, 0 \leq r < R
\end{eqnarray*}
to be used henceforth. This function of $r$ is called the maximum modulus function of $\mathbf{f}$. The next assertion contains an estimate for $\mathcal{M}(\mathbf{f},r)$ in terms of $|\mathbf{f}(0)|$, and the $C$-norm of the scalar part of $\mathbf{f}$ and one of its other components. This estimate generalizes the famous Hadamard-Borel-Carath\'{e}odory's theorem from complex one-dimensional analysis and refines the result of the authors in \cite{GueJoaoMappings2009}.

\begin{Theorem} [Real-Part theorem] \label{RealPartTheorem}
Let $\mathbf{f}$ be a square integrable $\mathcal{A}$-valued monogenic function in $B_R$. Then, for $0\leq r<R$ we have the inequality
\begin{eqnarray*}
\mathcal{M}(\mathbf{f},r) \, \leq \, |\mathbf{f}(0)| \, + \, \frac{2r}{(R-r)^2} \left(A_1(r,R)
\sup_{|\xi| < R} |\mathbf{Sc} \{\mathbf{f}(\xi)\}| \right. \\
\left. + \, A_2(r,R) \sup_{|\xi| < R} |\mathbf{Sc} \{\mathbf{h} \mathbf{e}_1(\mathbf{\xi})\}| \right),
\end{eqnarray*}
where
\begin{eqnarray*}
A_1(r,R) = \sqrt{\frac{2}{3}} \frac{(3R-r)(3R^2-Rr+r^2)}{(R-r)^2} \,\,\,\,\, {\rm and } \,\,\,\,\, A_2(r,R) = 2(2R-r). 
\end{eqnarray*}
\end{Theorem}
\begin{proof}
The proof is analogous to the one of Theorem 8. from \cite{GueJoaoMappings2009} with the estimates stated in Proposition \ref{modulusHMP}.
\end{proof}

\begin{Remark}
This theorem allows us to conclude that an $\mathcal{A}$-valued monogenic $L_2$-function $\mathbf{f}$ is bounded by a combination of its scalar part and one of its other components, without determining the exact value of $\mathcal{M}(\mathbf{f},r)$ which is very difficult or even impossible to obtain in most of the cases. As a general estimate it gives much more information for the local behavior of the maximum modulus of an $\mathcal{A}$-valued monogenic function but, in fact, it is not a complete analogy to the complex case. Therein, an analytic function is only bounded by its real part. On the other hand, a simple observation shows that analogously to the complex case the factor $1/(R - r)$ occurs in our constants and not $1/(R - 2r)$ as it was shown by the authors in \cite{GueJoaoMappings2009}. The reason is that the estimates stated in Proposition \ref{modulusHMP} refine the ones obtained in \cite{GueJoaoMappings2009}.
\end{Remark}

There are several hints and arguments to accept the class of $\mathcal{A}$-valued monogenic functions orthogonal to the subspace of non-trivial hyperholomorphic constants (in $L_2(B_R;\mathcal{A};\mathbb{R})$) as very well adapted to the class of holomorphic functions in the plane (for an account of such arguments, see e.g. \cite{GueJoaoMappings2009,GueJoaoBohr2} and \cite{JoaoThesis2009} Ch.3). This is the reason to restrict the latter to this subclass of functions. We get then a stronger result:

\begin{Corollary}% \label{RealPartTheoremOrthogonality}
Let $\tilde{\mathbf{f}}$ be a square integrable $\mathcal{A}$-valued monogenic function in $B_R$ orthogonal to the non-trivial hyperholomorphic constants with respect to the real-inner product $(\ref{InnerProduct})$. Then, for $0 \leq r<R$ we have the inequality:
\begin{eqnarray} \label{eqRealPartTheoremOrthogonality}
\mathcal{M}(\mathbf{\tilde{\mathbf{f}}},r) \, \leq \, |\tilde{\mathbf{f}}(0)| + \sqrt{\frac{2}{3}} \frac{2r A(r,R)}{(R-r)^4} 
\sup_{|\xi| < R} |\mathbf{Sc} \{\tilde{\mathbf{f}}(\xi)\}| ,
\end{eqnarray}
where $A(r,R) = (3R-r)(3R^2-Rr+r^2)$.
\end{Corollary}

\begin{Remark}
The previous corollary states that an $\mathcal{A}$-valued monogenic function $\tilde{\mathbf{f}} \in L_2(B_R;\mathcal{A};\mathbb{R}) \cap \ker D$ which is orthogonal to the non-trivial hyperholomorphic constants is bounded only by its scalar part. This underlines in an impressive way a complete analogy to the complex case, where $f-f(0)$ is also orthogonal to the constants. We remind the reader that in the higher dimensional case, we have to pay attention to the fact that the subspace of the hyperholomorphic constants is much bigger than only the subspace generated by $\mathbf{f}(0)$.
\end{Remark}

\begin{Remark}
In $(\ref{eqRealPartTheoremOrthogonality})$ when $r$ is sufficiently small, we have the local approximation's factor 
$\sqrt{\frac{2}{3}}\frac{18 r}{R}$, instead of $\frac{2r}{R}$ as it occurs in the complex case.
\end{Remark}

Replacing now $\tilde{\mathbf{f}}(x)$ by $\mathbf{f}(x) - \mathbf{f}(0)$ in the resulting relation from the previous corollary, we obtain a refinement of Hadamard's real part theorem.

\begin{Theorem} 
Let $\mathbf{f}$ be a square integrable $\mathcal{A}$-valued monogenic function in $B_R$ orthogonal to the non-trivial hyperholomorphic constants with respect to the real-inner product $(\ref{InnerProduct})$. Then, for $0 \leq r<R$ we have the inequality:
\begin{eqnarray*}
\mathcal{M}(\mathbf{f} - \mathbf{f}(0),r) \, \leq \, \sqrt{\frac{2}{3}} \frac{2r A(r,R)}{(R-r)^4} 
\sup_{|\xi| < R} |\mathbf{Sc} \{\mathbf{f}(\xi)-\mathbf{f}(0)\}|
\end{eqnarray*}
with $A(r,R)$ as in the previous corollary.
\end{Theorem}

\section{Estimates for $\mathcal{A}$-valued monogenic functions by the gradient}

A natural question in this context is to ask whether one can also obtain similar relations for the growth of an $\mathcal{A}$-valued monogenic function in terms of the growth of the gradient of its scalar part. The next theorem gives us a first result.

\begin{Theorem} \label{RealPartTheoremOrthogonality}
Let $\tilde{\mathbf{f}}$ be a square integrable $\mathcal{A}$-valued monogenic function in $B_R$ orthogonal to the non-trivial hyperholomorphic constants with respect to the real-inner product $(\ref{InnerProduct})$. Then, for $0 \leq r<R$ we have the following inequality:
\begin{eqnarray*}
\mathcal{M}(\mathbf{\tilde{\mathbf{f}}},r) \, \leq \, |\tilde{\mathbf{f}}(0)| 
\, + \, 4 \sqrt{2} \; \frac{r (2R-r)}{(R-r)^2} \sup_{|\xi| < R} |{\rm grad}(\mathbf{Sc} \{\tilde{\mathbf{f}}(\xi)\})|.
\end{eqnarray*}
\end{Theorem}
\begin{proof}
Let $\tilde{\mathbf{f}}$ be written as Fourier series
\begin{eqnarray*}
\tilde{\mathbf{f}} \, = \, \sum_{n=0}^{\infty} \left[ \mathbf{X}_n^{0,\dagger,\ast_R} a_n^{0,\ast_R} 
+ \sum_{m=1}^{n} \left( \mathbf{X}_n^{m,\dagger,\ast_R} a_n^{m,\ast_R} + \mathbf{Y}_n^{m,\dagger,\ast_R} b_n^{m,\ast_R} \right) \right],
\end{eqnarray*}
where for each $n \in \mathbb{N}_0$, $a_n^{0,\ast_R}, a_n^{m,\ast_R}, b_n^{m,\ast_R}$ $(m=1,...,n)$ are the associated Fourier coefficients. At first, notice that in the previous series the sum which contains the variable $m$ runs now only from $1$ to $n$. This fact expresses the supposed orthogonality to the hyperholomorphic constants $\mathbf{X}^{n+1,\dagger}_n$ and $\mathbf{Y}^{n+1,\dagger}_n$. We prove some relations between the Fourier coefficients of $\tilde{\mathbf{f}}$ and the Fourier coefficients associated to the gradient of its scalar part. By construction, the Fourier coefficients are real-valued and since the 
gradient is a linear operator, then it becomes clear that the gradient of $\textbf{Sc}(\tilde{\mathbf{f}})$ is defined in the following way
\begin{eqnarray*}
{\rm grad} (\textbf{Sc}(\tilde{\mathbf{f}})) &=& \sum_{n=0}^{\infty} \left\{ {\rm grad}(\textbf{Sc}(\mathbf{X}_n^{0,\dagger,\ast_R})) a_n^{0,\ast_R} \right. \\
&+& \left. \sum_{m=1}^{n} \left[ {\rm grad}(\textbf{Sc}(\mathbf{X}_n^{m,\dagger,\ast_R})) a_n^{m,\ast_R} 
+ {\rm grad}(\textbf{Sc} (\mathbf{Y}_n^{m,\dagger,\ast_R})) b_n^{m,\ast_R} \right] \right\}.
\end{eqnarray*}

Multiplying both sides of the previous expression by the (orthogonal) homogeneous anti-monogenic polynomials (see statement 4. Theorem \ref{PropertiesHMP})
\begin{eqnarray*}\label{solidsphericalharmonics}
\{{\rm grad}(\textbf{Sc}(\mathbf{X}_n^{0,\dagger,\ast_R})), {\rm grad}(\textbf{Sc}(\mathbf{X}_n^{m,\dagger,\ast_R})), 
{\rm grad}(\textbf{Sc}(\mathbf{Y}_n^{m,\dagger,\ast_R})) : m=1,...,n\}
\end{eqnarray*}
and integrating over $B_R$, we get the following relations:
\begin{eqnarray*}
a_n^{0,\ast_R} &=& \frac{\|\mathbf{X}_n^{0,\dagger}\|_{L_2(B_R;\mathcal{A};\mathbb{R})}}{\|{\rm grad}(\mathbf{Sc}(\mathbf{X}_n^{0,\dagger}))\|_{L_2(B_R;\mathcal{A};\mathbb{R})}^2} \int_{B_R} {\rm grad}(\mathbf{Sc}(\tilde{\mathbf{f}})) {\rm grad}(\mathbf{Sc}(\mathbf{X}_n^{0,\dagger})) dV_R \\
a_n^{m,\ast_R} &=& \frac{\|\mathbf{X}_n^{m,\dagger}\|_{L_2(B_R;\mathcal{A};\mathbb{R})}}{\|{\rm grad}(\mathbf{Sc}(\mathbf{X}_n^{m,\dagger}))\|_{L_2(B_R;\mathcal{A};\mathbb{R})}^2} \int_{B_R} {\rm grad}(\mathbf{Sc}(\tilde{\mathbf{f}})) {\rm grad}(\mathbf{Sc}(\mathbf{X}_n^{m,\dagger})) dV_R \\
b_n^{m,\ast_R} &=& \frac{\|\mathbf{Y}_n^{m,\dagger}\|_{L_2(B_R;\mathcal{A};\mathbb{R})}}{\|{\rm grad}(\mathbf{Sc}(\mathbf{Y}_n^{m,\dagger}))\|_{L_2(B_R;\mathcal{A};\mathbb{R})}^2} \int_{B_R} {\rm grad}(\mathbf{Sc}(\tilde{\mathbf{f}})) {\rm grad}(\mathbf{Sc}(\mathbf{Y}_n^{m,\dagger})) dV_R,
\end{eqnarray*}
for $m=1,\ldots,n$. We remark that originally the Fourier coefficients are defined by the inner product of the function $\tilde{\mathbf{f}}$ and elements of the space $\mathcal{M}^+(\mathbb{R}^3;\mathcal{A};n)$. Now we see that these coefficients, up to a factor, can also be expressed as inner products of the gradient of the scalar part of $\tilde{\mathbf{f}}$ and ${\rm grad}(\mathbf{Sc}(\mathbf{X}_n^{0,\dagger})), {\rm grad}(\mathbf{Sc}(\mathbf{X}_n^{m,\dagger}))$ and ${\rm grad}(\mathbf{Sc}(\mathbf{Y}_n^{m,\dagger}))$, respectively.

Considering $\tilde{\mathbf{f}}$ written as in $(\ref{decomposition})$ and taking into account the maximum modulus principle we have
\begin{eqnarray*}
|\tilde{\mathbf{f}}|_r \leq |\tilde{\mathbf{f}}(0)| + |\mathbf{g}|_R.
\end{eqnarray*}

Using the previous estimates of the Fourier coefficients it follows that
\begin{eqnarray*}
|\mathbf{g}|_R \, &\leq& \, 2\sqrt{\frac{\pi}{3}} \sqrt{R^3} \sup_{|\xi| < R} |{\rm grad}(\mathbf{Sc} \{\tilde{\mathbf{f}}(\xi)\})| 
\sum_{n=1}^{\infty} \left[ \frac{|\mathbf{X}_n^{0,\dagger,\ast_R}| \, \|\mathbf{X}_n^{0,\dagger}\|_{L_2(B_R;\mathcal{A};\mathbb{R})}}{\|{\rm grad}(\mathbf{Sc}(\mathbf{X}_n^{0,\dagger}))\|_{L_2(B_R;\mathcal{A};\mathbb{R})}}
\right. \\
&+& \left. \sum_{m=1}^{n} \left( \frac{|\mathbf{X}_n^{m,\dagger,\ast_R}| \, \|\mathbf{X}_n^{m,\dagger}\|_{L_2(B_R;\mathcal{A};\mathbb{R})}}{\|{\rm grad}(\mathbf{Sc}(\mathbf{X}_n^{m,\dagger}))\|_{L_2(B_R;\mathcal{A};\mathbb{R})}} 
+ \frac{|\mathbf{Y}_n^{m,\dagger,\ast_R}| \, \|\mathbf{Y}_n^{m,\dagger}\|_{L_2(B_R;\mathcal{A};\mathbb{R})}}{\|{\rm grad}(\mathbf{Sc}(\mathbf{Y}_n^{m,\dagger}))\|_{L_2(B_R;\mathcal{A};\mathbb{R})}}
\right) \right] .
\end{eqnarray*}\

Finally, applying Proposition \ref{modulusHMP} we obtain
\begin{eqnarray*}
|\tilde{\mathbf{f}}|_r \, \leq \, |\tilde{\mathbf{f}}(0)| \, + \, 4\sqrt{2} \sup_{|\xi| < R} |{\rm grad}(\mathbf{Sc} \{\tilde{\mathbf{f}}(\xi)\})| 
\sum_{n=1}^{\infty} \left(\frac{r}{R}\right)^n (n+1).
\end{eqnarray*}

Now, note that the last series is convergent for $0 \leq r < R$ and the theorem is proved.
\end{proof}

\section{Estimates for derivatives of $\mathcal{A}$-valued monogenic functions}

As in the case of holomorphic functions in the complex plane we have to wonder if also the growth of the derivative (here replaced by the hypercomplex derivative) can be bounded by the growth of the function. If this is possible, consequently, the behavior of the hypercomplex derivative can also be estimated by the scalar part of the $\mathcal{A}$-valued monogenic function. The following theorem refines a corresponding result in \cite{GueJoaoMappings2009}.

\begin{Theorem}
Let $\mathbf{f}$ be a square integrable $\mathcal{A}$-valued monogenic function in $B_R$. Then, for $0\leq r< R$ we have the following inequality:
\begin{eqnarray*}
\mathcal{M}\left((\frac{1}{2} \overline{D}) \mathbf{f},r\right) ~ \leq ~ 2\sqrt{3} \,
\frac{rR (5r+4R)}{(R-r)^4} \sup_{|\xi| < R} |{\bf{Sc}}\{\mathbf{f}(\xi)\}| .
\end{eqnarray*}
\end{Theorem}
\begin{proof}
The proof is analogous to the one of Theorem 9. from \cite{GueJoaoMappings2009} with the estimates stated in Proposition \ref{modulusHMP} and taking into account the equalities for the homogeneous monogenic polynomials and their hypercomplex derivatives (see Property 6. of Theorem \ref{PropertiesHMP}).
\end{proof}

\end{document}